\documentclass[a4paper,12pt,leqno]{amsart}
\usepackage{geometry}
\usepackage[utf8]{inputenc}
\usepackage{amssymb,amsfonts,amsmath,amsthm,eucal,url,cite}
\usepackage{enumitem}
\usepackage{tikz-cd}
\usepackage{hyperref}
\usepackage{multicol}
\usepackage{slashed}

\newcommand{\RR}{\mathbb{R}}
\newcommand{\CC}{\mathbb{C}}

\newcommand{\NN}{\mathbb{N}}
\newcommand{\ZZ}{\mathbb{Z}}
\newcommand{\HH}{\mathbb{H}}

\def\SS{\mathbb S}

\newtheorem{theorem}{Theorem}[section]
\newtheorem{prop}[theorem]{Proposition}
\newtheorem{cor}[theorem]{Corollary}

\newtheorem{definition}[theorem]{Definition}

\theoremstyle{definition}

\theoremstyle{remark}
\newtheorem{rem}[theorem]{\bf Remark}
\newtheorem{example}[theorem]{\bf Example}

\title{On decomposable LCP structures}
\author{Brice Flamencourt, Andrei Moroianu}

\address{Brice Flamencourt, Universität Stuttgart, Institut für Geometrie und Topologie, Fachbereich Mathematik, Pfaffenwaldring 57, 70569 Stuttgart, Germany.}
\email{brice.flamencourt@mathematik.uni-stuttgart.de}

\address{Andrei Moroianu \\ Université Paris-Saclay, CNRS,  Laboratoire de mathématiques d'Orsay, 91405, Orsay, France, 
and Institute of Mathematics “Simion Stoilow” of the Romanian Academy, 21 Calea Grivitei, 010702 Bucharest, Romania}
\email{andrei.moroianu@math.cnrs.fr}


\begin{document}
\begin{abstract}
    We introduce the notion of decomposable locally conformally product (LCP) manifolds and characterize those which are defined on quotients of Riemannian Lie groups by co-compact lattices.
\end{abstract}

\maketitle

\section{Introduction}

The study of locally conformally product (LCP) structures is a topic of conformal geometry which has grown up in the last decade, starting with the works of Belgun-Moroianu \cite{BM}, Matveev-Nikolayevsky \cite{MN15,MN17} and Kourganoff \cite{Kou} and more recently developped by Andrada-del Barco-Moroianu \cite{ABM,dBM}, Flamencourt \cite{FlaLCP,F24}, and other authors \cite{BFM, MP}.

LCP manifolds are in many respects similar to locally conformally Kähler (LCK) manifolds. They can be defined either as compact quotients of simply connected (non-flat) Riemannian manifolds with reducible holonomy by discrete subgroups of homotheties not containing only isometries, or as compact conformal manifolds $(M,c)$ carrying a closed non-exact Weyl connection $\nabla$ with reducible (non-zero) holonomy.

According to a fundamental result of Kourganoff \cite{Kou}, the tangent bundle of an LCP manifold $(M,c,\nabla)$ splits into two $\nabla$-parallel distributions, one of which is flat. A Riemannian metric $g$ on $M$ in the conformal class $c$ is called {\em adapted} if the Lee form of $\nabla$ with respect to $g$ vanishes on the flat distribution. Adapted metrics always exist \cite{FlaLCP,MP} and their importance is given by the following observation \cite{FlaLCP}: If $(M,c,\nabla)$ is an LCP structure, and $g\in c$ is adapted, then for every compact Riemannian manifold $(K,g_K)$, the conformal manifold $(M \times K, [g + g_K])$ carries an adapted LCP structure as well. 

The LCP structures obtained in this way were called {\em reducible} in \cite{FlaLCP}, and it is obvious that the study of LCP structures can be reduced to understanding the irreducible ones. However, it might happen that an irreducible LCP manifold is {\em weakly reducible}, in the sense that it is obtained from a reducible LCP manifold by changing the action of the fundamental group on the universal cover (cf. \cite[Example 4.12]{BFM}). 

Because of this phenomenon, we introduce below the slightly more general notion of {\em decomposable} LCP structure (which is, by definition, an LCP structure containing a Riemannian metric with reducible holonomy in the conformal class $c$). One of the motivations for this definition is \cite[Theorem 4.7]{BFM} where it is shown that a decomposable LCP structure $(M,c,\nabla)$ is locally (but in general not globally) reducible.

The article is structured as follows: first of all we recall the basics about LCP manifolds in Section~\ref{preliminaries}, and we review some background material concerning Riemannian Lie groups. We also prove a result of independent interest concerning the de Rham decomposition of simply connected Riemannian Lie groups in Theorem~\ref{2.5}, namely that the factors in the decomposition can be taken to be Lie subgroups of the ambient manifold. Even though this result is probably known to specialists in the field, we were not able to find a reference for it, so we thought it was worth giving it a proof. 

In Section~\ref{generalLiegroups} we investigate reducibility properties concerning Riemannian Lie groups and homogeneous spaces. We start by obtaining a criterion to check the holonomy reducibility of a Riemannian Lie group in Proposition~\ref{redg}, which will be useful for later applications. Next, our attention focuses on conformal geometry of homogeneous spaces - even though our applications will always concern Lie groups in this article -, and we derive the really convenient Theorem~\ref{conformLiegroup}, stating that the only metric which can be reducible in the conformal class of such Riemannian manifolds is the homogeneous one, up to a multiplicative constant.

Section~\ref{Sectiondecomp} is devoted to the main subject of this paper, decomposable LCP manifolds. In Proposition \ref{q+2} we extend \cite[Theorem 4.7]{BFM} (which states that the universal cover of a decomposable LCP manifold has a de Rham factor containing the flat distribution $\RR^q$ and the metric dual of the Lee form $\theta$) by showing that this factor has dimension at least $q+2$, i.e. is strictly larger than $\RR^q\oplus \theta^\sharp$.

We then investigate a particular subclass of LCP structures: the Lie LCP manifolds, which are compact quotients of Riemannian Lie groups by lattices, focusing on a characterization of their decomposability.  Notice that proving the indecomposability of a general LCP manifold is far from being trivial, since one has to check the irreducibility of each metric in the conformal class. Luckily, in Corollary~\ref{declcp} we show that the only possibly reducible metric on a Lie LCP manifold is the left-invariant one up to a multiplicative constant, so their decomposability can be described in purely algebraic terms by the metric properties of the Lie algebras.

Our previous results converge to one application, namely showing that a fundamental example of Lie LCP manifold (constructed in \cite[Section 5.2]{dBM}), is indecomposable, a property proved in Proposition~\ref{notdecomposable}. We end the paper with a last section discussing some refinements of the notion of reducibility for LCP manifolds and giving further examples of LCP manifolds which belong to the different classes of decomposable LCP structures introduced here.

{\bf Acknowledgments.} This work was partly supported by the PNRR-III-C9-2023-I8 grant CF 149/31.07.2023 {\em Conformal Aspects of Geometry and Dynamics} and by the Procope Project No. 57650868 (Germany) / 48959TL (France).

\section{Preliminaries} \label{preliminaries}

\subsection{LCP structures}

We recall here the basic definitions concerning Locally Conformally Product (in short, LCP) structures. A detailed discussion on this topic can also be found in \cite{FlaLCP} for example.

In order to define an LCP structure, we first need to introduce Weyl connections, which are special connections in conformal geometry, generalizing the concept of Levi-Civita connection from Riemannian geometry.

\begin{definition}
Let $(M, c)$ be a conformal manifold. A Weyl connection on $(M,c)$ is a torsion-free connection $\nabla$ on $M$ which preserves the conformal structure, i.e. for any metric $g \in c$, there is a $1$-form $\theta_g$, called the {\em Lee form} of $\nabla$ with respect to $g$, such that $\nabla g = - 2 \theta_g \otimes g$.
 \end{definition}

The Lee form of a given Weyl connection $\nabla$ depends on the metric in the conformal class. However, the Lee forms of $\nabla$ with respect to any two metrics in $c$ differ by an exact $1$-form, which motivates the following definition:

\begin{definition}
A Weyl connection $\nabla$ on a conformal manifold $(M,c)$ is called {\em closed} if the Lee form of $\nabla$ with respect to one metric - and then to all metrics - in $c$ is closed. Similarly, $\nabla$ is called {\em exact} if the Lee form of $\nabla$ with respect to one metric - and then to all metrics - in $c$ is exact.
\end{definition}

Moreover, the Lee form gives information on the nature of $\nabla$, as shown by the following fundamental property:

\begin{prop}
Let $(M,c)$ be a conformal manifold endowed with a Weyl connection $\nabla$. If $\nabla$ is closed, then it is locally the Levi-Civita connection of a metric in $c$. If $\nabla$ is exact, then this statement holds globally.
\end{prop}

In the case of a closed Weyl connection $\nabla$ on a conformal manifold $(M,c)$, the pull-back $\tilde \nabla$ of the Weyl connection to the universal covering $\tilde M$ of $M$ is a Weyl connection for the conformal structure $\tilde c$ obtained by pulling-back $c$. This Weyl connection is exact since $\tilde M$ is simply connected, thus there exists a metric $h \in \tilde c$, unique up to a multiplication by a constant, such that $\nabla^h = \tilde \nabla$, where $\nabla^h$ is the Levi-Civita connection of $h$. This metric is invariant by the fundamental group of $M$ (i.e. it is the pull-back of a metric on $M$) if and only if $\nabla$ is exact.  

LCP structures arise when one consider closed, non-exact Weyl connections on a compact conformal manifold. In this situation, one has a remarkable result proved by Kourganoff \cite{Kou}:

\begin{theorem}[Kourganoff] \label{fundLCP}
Let $(M,c)$ be conformal manifold endowed with a closed, non-exact Weyl connection $\nabla$. Let $h$ be a metric on $\tilde M$, the universal cover of $M$, such that $\nabla^h = \tilde\nabla$ where $\tilde \nabla$ is the pull-back of $\nabla$ to $\tilde M$. Then, one of the three following cases occurs:
\begin{itemize}
\item $(\tilde M, h)$ is flat;
\item $(\tilde M, h)$ is irreducible;
\item $(\tilde M, h)$ is a Riemannian product $\RR^q \times (N, g_N)$ where $q \ge 1$ and $(N, g_N)$ is a non-flat, incomplete Riemannian manifold.
\end{itemize}
\end{theorem}

The third case in Theorem~\ref{fundLCP} corresponds to so-called LCP structures. More precisely:

\begin{definition}
An LCP structure is a triple $(M,c,\nabla)$ where $M$ is a compact manifold, $c$ is a conformal structure on $M$ and $\nabla$ is a closed, non-exact Weyl connection, which is non-flat and reducible (i.e. the representation of its restricted holonomy group $\mathrm{Hol}_0(\nabla)$ is reducible).
\end{definition}

With the notations of the third case of Theorem~\ref{fundLCP},  $\RR^q$ is called the {\em flat part} of the LCP structure, while $(N,g_N)$ is called the non-flat part. The distributions $T \RR^q$ and $T N$ descend to $\nabla$-parallel distributions on $M$, respectively called the {\em flat distribution} and the {\em non-flat distribution} of the LCP manifold.

\subsection{Riemannian Lie groups}
Let $G$ be a Lie group. In all this paper, Lie groups are considered to be connected (except in the very particular case where we consider a lattice in a Lie group).

Left-invariant objects on $G$ are completely described by their counterpart on the Lie algebra $\mathfrak g$ of $G$. In this spirit, left-invariant vector fields will be viewed as elements of $\mathfrak g$. Moreover, we recall that the structure of a {\em Riemannian} Lie group is given by a left-invariant metric $g$ on $G$, i.e. such that the left-translations act as isometries. This metric is then completely determined by a positive definite scalar product $\langle \cdot, \cdot \rangle$ on $\mathfrak g$. The Levi-Civita connection $\nabla^g$ of $g$ preserves left-invariant vector fields, so it can also be viewed as a bilinear map $\nabla^g: \mathfrak g \times \mathfrak g \to \mathfrak g$ determined by the Koszul formula:
\begin{equation} \label{LCliegroup}
2 \langle \nabla^g_x y, z \rangle = \langle [x,y],z \rangle + \langle [z,x],y \rangle - \langle [y,z],x \rangle.
\end{equation}

We discuss in this section some general results concerning Riemannian Lie groups, that we will later apply to the particular case of LCP Lie manifolds. More precisely, we turn our attention to the reducibility of the holonomy group of these Riemannian manifolds. The de Rham decomposition of a simply connected Riemannian Lie group admits an interesting structure, which we highlight here:

\begin{theorem}\label{2.5}
Let $(G, g)$ be a simply connected Riemannian Lie group. Then, the factors in the de Rham decomposition of $(G,g)$ can be taken to be subgroups of $G$ with the induced metrics.
\end{theorem}
\begin{proof}
The decomposition of the holonomy representation of $\nabla^g$ induces $\nabla^g$-parallel distributions $D_0, \ldots, D_k$ on $G$ such that
\begin{equation}\label{dr}T G = D_0 \overset{\perp}{\oplus} \ldots \overset{\perp}{\oplus} D_k,
\end{equation}
where $D_0$ is the maximal flat distribution. Correspondingly, by the global de Rham theorem, $(G, g)$ is isometric to the product $(G_0, g_0) \times \ldots \times (G_k, g_k)$ where $G_i$ is the integral leaf of the distribution $D_i$ passing through the identity. Using the fact that left-translations are isometries of $(G,g)$, together with the uniqueness of the de Rham decomposition (up to permutation of the factors) and the connectedness of $G$, it turns out that for any $a \in G$, $(L_a)_* D_i = D_i$.

Then, for $a \in G_i$, $L_a G_i$ contains $a$ and is an integral leaf of $D_i$. Thus $L_a G_i =  G_i$, and for any $b \in G_i$ we have that $a b \in G_i$. In addition, there exists $b \in G_i$ such that $a b = e \in G_i$, showing that $a^{-1} = b \in G_i$. Altogether, we proved that the $G_i$'s are subgroups of $G$ and the metrics $g_i$ for $i \in \{0, \ldots, k\}$ are obviously left-invariant.
\end{proof}

\begin{rem}\label{rdr}
    The same argument shows that more generally, for every subset $I\subset \{0,\ldots,k\}$ of the set of indices of the de Rham splitting \eqref{dr}, the distribution $T^I:=\oplus_{i\in I}D_i$ is integrable, left-invariant, and its integral leaf through the identity is a subgroup of $G$.
\end{rem}

\section{Holonomy reducibility of Riemannian Lie groups and homogeneous spaces} \label{generalLiegroups}

\subsection{Reducible Riemannian Lie groups}

A criterion to detect Riemannian Lie groups with reducible holonomy through their Lie algebra can now be formulated under the following form:

\begin{prop}\label{redg}
    Let $(G, g)$ be a simply connected Riemannian Lie group and denote by $(\mathfrak g, \langle \cdot , \cdot \rangle)$ the corresponding metric Lie algebra. Then, $(G,g)$ has reducible holonomy as a Riemannian manifold if and only if there exists a non-trivial orthogonal decomposition $\mathfrak g = \mathfrak g_1{\oplus} \mathfrak g_2$ such that $\mathfrak g_1$ and $\mathfrak g_2$ are Lie subalgebras of $\mathfrak g$ and
\begin{equation}\label{cond}
\langle [x_1,x_2], x_2 \rangle = \langle [x_1, x_2], x_1 \rangle = 0, \qquad \forall x_1\in \mathfrak g_1,\ \forall x_2\in \mathfrak g_2.
\end{equation}
\end{prop}
Note that, by polarization, \eqref{cond} is equivalent to 
\begin{equation}\label{cond1}
\langle [x_1,x_2], y_2 \rangle+\langle [x_1,y_2], x_2 \rangle = 0, \qquad \forall x_1\in \mathfrak g_1,\ \forall x_2,y_2\in\mathfrak g_2,
\end{equation}
\begin{equation}\label{cond2}
\langle [x_1, x_2], y_1 \rangle+\langle [y_1, x_2], x_1 \rangle = 0, \qquad \forall x_1,y_1\in \mathfrak g_1,\ \forall x_2\in\mathfrak g_2.
\end{equation}
\begin{proof}
Assume that $(G,g)$ has reducible holonomy and let $TG=T_1\oplus T_2$ be an orthogonal $\nabla^g$-parallel decomposition of the tangent bundle of $G$, such that each of $T_1$ and $T_2$ are direct sums of distributions of the de Rham splitting \eqref{dr} of $G$. By Theorem \ref{2.5} and Remark \ref{rdr}, $T_1$ and $T_2$ are left-invariant distributions, and their the integral leaves through the identity are Lie subgroups of $G$.
We denote these subgroups by $G_1$ and $G_2$, and their Lie algebras by $\mathfrak{g}_1$  and $\mathfrak{g}_2$. For every $x_1\in \mathfrak{g}_1$ and  $x_2\in \mathfrak{g}_2$ one has:
\[\langle [x_1,x_2], x_1 \rangle = \langle \nabla^g_{x_1} x_2-\nabla^g_{x_2} x_1, x_1\rangle =0,\]
since $\nabla^g_{x_1} x_2$ is in $\mathfrak g_2$ and by use of Equation~\eqref{LCliegroup}. The second relation in \eqref{cond} is similar.

Conversely, assume that $\mathfrak g = \mathfrak g_1{\oplus} \mathfrak g_2$ is an orthogonal direct sum decomposition of $\mathfrak{g}$ into subalgebras satisfying \eqref{cond}. Let $T_1$ and $T_2$ be the left-invariant distributions of $G$ determined by $\mathfrak g_1$ and $\mathfrak g_2$. For every $x_1,y_1\in \mathfrak{g}_1$ and  $x_2\in \mathfrak{g}_2$, the Koszul formula~\eqref{LCliegroup}, shows that
\begin{align*}
2 \langle \nabla^g_{x_1} x_2, y_1 \rangle &= \langle [x_1, x_2], y_1 \rangle + \langle [y_1, x_1], x_2 \rangle - \langle [x_2, y_1], x_1 \rangle \\
&= \langle [x_1,x_2], y_1 \rangle+\langle [y_1,x_2], x_1 \rangle\stackrel{\eqref{cond2}}{=}0.
\end{align*}
Similarly we obtain that $\langle \nabla^g_{x_2} x_1, y_2 \rangle=0$ for all $x_1 \in \mathfrak g_1$ and $x_2,y_2 \in \mathfrak g_2$. This shows that $T_1$ and $T_2$ are $\nabla^g$-parallel.
\end{proof}

\subsection{Reducible metrics on homogeneous spaces}

We investigate in this section the reducibility of complete metrics in the conformal class of homogeneous Riemannian metrics. Our goal is to prove that only the constant multiples of the homogeneous metric can be complete and reducible.
Notice that in the compact case, this result follows from the more general fact that non-constant multiples of homogeneous metrics on compact simply connected homogeneous manifolds of dimension $n$ have holonomy $\mathrm{SO}(n)$ (cf. \cite[Corollary 2.3]{MS2008}).

We first recall a remarkable result of Tashiro and Miyashita \cite{TM67}, which gives a strong obstruction for complete Riemannian products to admit non-isometric conformal vector fields.

\begin{theorem}[\cite{TM67}] \label{TashMiy}
Every complete conformal vector field on a complete non-flat Riemannian manifold with reducible holonomy is Killing.
\end{theorem}

Assuming the existence of complete reducible metrics in the conformal class of a homogeneous metric, and using the fact that for any point $x$ there is a family of Killing vector fields forming a basis of the tangent space at $x$, our result will follow from Theorem~\ref{TashMiy}, provided that we know which homogeneous spaces are globally conformally flat. We thus need a classification of these spaces.

We start by proving the following rigidity property of the standard metric on $\RR^n$ in its conformal class:

\begin{theorem} \label{conformRn}
Let $n \ge 3$ and let $g$ be a scalar-flat Riemannian metric on $\RR^n$, belonging to the conformal class of the standard metric $g_{\RR^n}$. Then $g$ is homothetic to $g_{\RR^n}$. In particular, $g_{\RR^n}$ is the unique flat metric in $[g_{\RR^n}]$ up to a multiplication by a constant.
\end{theorem}
\begin{proof}
We write $g = e^{2f} g_{\RR^n}$ for some smooth function $f : \RR^n \to \RR$ and denote by $\Delta$ the Laplacian of $g_{\RR^n}$. The conformal change formula for the scalar curvature \cite[\S 1.159]{besse} gives:
\begin{equation}
0 =(n-1) e^{-2f}(2\Delta f-(n-2)|df|^2)=e^{-2f} \frac{4(n-1)}{(n-2)} e^{-(n-2) f/2} \Delta \left( e^{(n-2) f/2} \right),
\end{equation}
implying
\begin{equation}
\Delta \left( e^{(n-2) f/2} \right) = 0.
\end{equation}
This means that $e^{(n-2) f/2}$ is a positive harmonic function on $\RR^n$, thus it is constant, showing that $g$ is homothetic to $g_{\RR^n}$.
\end{proof}

A consequence of Theorem~\ref{conformRn} is a classification of homogeneous spaces globally conformal to the Euclidean space $\RR^n$:

\begin{prop} \label{simplyconnectedHspace}
Let $(H, g_H)$ be a simply connected homogeneous space of dimension $n \ge 3$. Assume there is a metric $g \in [g_H]$ such that $(H, g)$ is isometric to $\RR^n$. Then, $g$ is homothetic to $g_H$. In particular, $(H, g_H)$ is isometric to $\RR^n$.
\end{prop}
\begin{proof}
The classification of (locally) conformally flat simply connected homogeneous spaces provided by \cite[Theorem 1]{AK78} implies that $(H, g_H)$ is homothetic to $\RR^n$, $\SS^n$, $\HH^n$, $\RR \times \SS^{n-1}$ or $\RR \times \HH^{n-1}$ since $(H, g_H)$ is globally conformally flat. Moreover, $H$ is diffeomorphic to $\RR^n$ by assumption, so the only possibilities in this list are $\RR^n$, $\HH^n$ and $\RR \times \HH^{n-1}$ for topological reasons.

{\bf Case 1: $(H, g_H)$ is homothetic to $\RR^n$.} Then, $g$ being a flat metric it is homothetic to the standard metric on $\RR^n$ by Theorem~\ref{conformRn} and thus homothetic to $g_H$.

{\bf Case 2: $(H, g_H)$ is homothetic to $\HH^n$.} We consider the model $\HH^n = \RR^{n-1} \times \RR^*_+$ with the metric being $\frac{1}{x_n^2} (dx_1^2 + \ldots + dx_n^2)$ in the standard coordinates. Then, $\HH^n$ is conformal to the flat manifold $\RR^{n-1} \times \RR^*_+$ with the standard metric. If $\HH^n$ were conformal to $\RR^n$, then $\RR^n$ would be conformal to $\RR^{n-1} \times \RR^*_+$, and by Theorem~\ref{conformRn} these two spaces would be homothetic. However, one is complete and not the other, which is a contradiction. Thus $(H, g_H)$ is not homothetic to $\HH^n$.

{\bf Case 3: $(H, g_H)$ is homothetic to $\RR \times \HH^{n-1}$.} Taking the same model as before for $\HH^{n-1}$, the metric on $\RR \times \HH^{n-1}$ is conformal to $x_n^2 dx_1^2 + dx_2^2 + \ldots + dx_n^2$. The Riemannian manifold $(\RR^2, dx_n^2 + x_n^2 dx_1^2)$ is the universal cover $\widetilde{\CC}^*$ of $\CC^*$ with the lift of the standard metric, thus it is a flat incomplete manifold. As in the previous case, $\RR^n$ would be homothetic to a flat incomplete manifold, which is a contradiction. Thus $(H, g_H)$ is not homothetic to $\RR \times \HH^{n-1}$.
\end{proof}

We are now able to prove the general result about conformal classes of homogeneous metrics mentioned above.

\begin{theorem} \label{conformLiegroup}
Let $(H,g)$ be a Riemannian homogeneous space. Assume there is a reducible and complete metric $g_0$ in the conformal class $[g]$. Then, $g_0$ is homothetic to $g$. Equivalently, $g$ is the unique possibly reducible complete metric in $[g]$, up to a multiplicative constant.
\end{theorem}
\begin{proof}
We consider the universal cover $\tilde H$ of $H$ and the two lifts $\tilde g$ and $\tilde g_0$ of the metrics $g$ and $g_0$ respectively.

Assume first that the metric $\tilde g_0$ is non-flat. Let $f \in C^\infty(H)$ such that $g_0 = e^{2f} g$, and denote by $\tilde f$ its lift to $\tilde H$. Then, $(\tilde H, \tilde g)$ is a complete Riemannian product. Every Killing vector field $\xi$ with respect to $\tilde g$ is complete and conformal with respect to $\tilde g_0$. Theorem~\ref{TashMiy} thus implies that $\xi$ is a Killing vector filed with respect to $\tilde g_0$. Therefore, one has
\[
0 = \mathcal L_\xi (e^{2 \tilde f} \tilde g) = 2\xi (\tilde f) \tilde g,
\]
and this yields $\xi (\tilde f) = 0$. This equality being true for any Killing vector field, and since each point of $\tilde H$ has a basis of (global) Killing vector fields, we conclude that $d f = 0$, i.e. $f$ is constant.

Assume now that $g_0$ is flat. Then, by Theorem~\ref{simplyconnectedHspace} $(\tilde H, \tilde g)$ is homothetic to $\RR^n$ and $\tilde g$ is homothetic to $\tilde g_0$, implying that $g$ is homothetic to $g_0$.
\end{proof}

\section{Decomposable LCP manifolds} \label{Sectiondecomp}

\subsection{Decomposable LCP structures} In \cite{FlaLCP} were introduced the so-called {\em reducible} LCP structures, which are LCP structures such that the conformal class $c$ contains a metric $g$ for which $(M,g)$ is a Riemannian product. Nevertheless, in view of \cite[Theorem 4.7]{BFM}, a slightly more general definition appears to be more relevant:

\begin{definition}
Let $(M,c)$ be a conformal manifold. We denote by $\mathcal D(M,c)$ the set of all metrics $g \in c$ with reducible holonomy, or equivalently, such that $TM$ carries a non-trivial $\nabla^g$-parallel distribution $0 \subsetneq D \subsetneq TM$.
\end{definition}

\begin{definition}\label{defdec}
An LCP manifold $(M,c,\nabla)$ is {\em decomposable} if $\mathcal D(M,c) \neq \emptyset$.
\end{definition}

The first examples of decomposable LCP structures are, as mentioned above, defined by reducible LCP manifolds. We recall here how to construct them.

\begin{example} \label{reducibleLCP}
Let $(M_1, c, \nabla_1)$ be an LCP manifold. Let $M_2$ be a compact manifold of positive dimension. Let $g_1 \in c$ be an {\em adapted} metric in the sense of \cite[Section 3.1]{FlaLCP}, i.e. such that the Lee form $\theta$ of $\nabla_1$ with respect to $g_1$ vanishes on the flat distribution of the LCP structure. Then, $\nabla_1 = \nabla^{g_1} + \bar \theta$, where $\bar \theta$ is the vector-valued bilinear form defined by
\begin{align} \label{defbar}
\bar \theta _X (Y) := \theta(X) Y + \theta (Y) X - g_1(X,Y) \theta^\sharp, && \forall X, Y \in TM.
\end{align}
For any Riemannian metric $g_2$ on $M_2$, the triple $(M := M_1 \times M_2, [g := g_1 + g_2], \nabla^g + \overline{\pi_1^* \theta})$ is an LCP structure, where $\pi_1:M\to M_1$ is the projection on the first factor $M_1$ and $\overline{\pi_1^* \theta}$ is the $(2,1)$-tensor determined by $\pi_1^* \theta$ as in \eqref{defbar}. This LCP structure is called {\em reducible}, and it is decomposable since $TM_1$ is a $\nabla^g$-parallel non-trivial distribution of $M$.
\end{example}

\begin{rem}
Not all decomposable LCP manifolds are reducible. An example of decomposable and irreducible LCP manifold is given in \cite[Example 4.12]{BFM}. The difference between the two notions is related to the action of the fundamental group on the universal cover. More precisely, by  \cite[Theorem 4.7]{BFM}, if $(M,g)$ has a decomposable LCP structure, then the universal cover $(\tilde M,\tilde g)$ is isometric to a Riemannian product $(\tilde M_1,\tilde g_1)\times (\tilde M_2,\tilde g_2)$, such that the flat factor is contained in $TM_1$ and the Lee form vanishes on $TM_2$. The action on $\tilde M$ of the fundamental group $\pi_1(M)$ preserves the product structure, that is, every $\gamma\in\pi_1(M)$ is of the form $(\gamma_1,\gamma_2)$ with $\gamma_i\in\mathrm{Iso}(\tilde M_i,\tilde g_i)$ for $i\in\{1,2\}$. If  $\pi_1(M)$ is a product $\Gamma_1\times\Gamma_2$ with $\Gamma_i\subset\mathrm{Iso}(\tilde M_i,\tilde g_i)$, then the LCP structure on $(M,g)$ is reducible. If not, it is only decomposable.
\end{rem}

In \cite{BFM}, the structure of decomposable LCP manifolds $(M,c,\nabla)$ was studied, showing that every metric $g\in\mathcal D(M,c)$ is adapted, and both the metric dual of the Lee form of $\nabla$ with respect to $g$, and the flat distribution, are tangent to one of the two $\nabla^g$-parallel distributions of $TM$. 

This result admits an easy but important corollary. In order to state it,  let $(M, c, \nabla)$ be a decomposable Lie LCP manifold and let $g \in \mathcal D(M,c)$. Since $M$ is compact, the pull-back metric $\tilde g$ to the universal cover $\tilde M$ is complete, so by the global de Rham decomposition theorem, 
$(\tilde M, \tilde g)$ is isometric to a product of complete simply connected Riemannian manifolds $(M_0, g_0) \times \ldots \times (M_k, g_k)$, $k \in \NN$, where $(M_0, g_0)$ is flat and the other factors are non-flat and irreducible.

\begin{theorem} \label{main1}
With the notations above, there exists $i \in \{ 1, \ldots, k \}$ such that, if $\RR^q$ is the flat part of the LCP manifold $(M, c, \nabla)$ and $\theta$ is the Lee form of $\nabla$ with respect to $g$, then $T \RR^q \oplus \RR \tilde \theta^\sharp \subset T M_i$ (where $\tilde \theta$ is the lift of $\theta$ to the universal cover of $M$).
\end{theorem}
\begin{proof}
 For every $i \in \{0, \ldots, k\}$, applying \cite[Theorem 4.7]{BFM} to the decomposition $T \tilde M = D_i \oplus D_i^\perp$ yields that $T \RR^q$ and $\tilde \theta^\sharp$ are either contained in $D_i$ or in $D_i^\perp$. Thus, there exists $i_0$ such that $T \RR^q$ and $\tilde \theta^\sharp$ are in $D_{i_0}$, because otherwise $\tilde \theta^\sharp$ would be in $D_i^\perp$ for all $i$, which is a contradiction with the fact that $\tilde\theta$ is non-zero and $\bigcap_{0 \le i \le n} D_i^\perp = \{ 0 \}$.

Assume that $i_0 = 0$. Since $(M_0, g_0)$ is a complete simply connected flat manifold, one can write $M_0 = d_1 \times \ldots \times d_p$ with $d_j$ being of dimension $1$. Applying \cite[Theorem 4.7]{BFM}  again, one has that for every $j\in\{1,\ldots,p\}$, $T \RR^q \oplus  \RR \tilde \theta^\sharp \subset T d_j \textrm{ or } T \RR^q \oplus  \RR \tilde \theta^\sharp \subset Td_j^\perp$. The first inclusion being impossible for dimensional reasons, the second inclusion has to hold for all $j$, leading to a contradiction since $\bigcap_{1 \le i \le p} T d_j^\perp = \{ 0 \}$. Thus, $i_0 \in \{1, \ldots, k\}$.\qedhere
\end{proof}

The de Rham factor defined in Theorem~\ref{main1},(2) contains almost all the information about the LCP structure, so we give it a name for later purposes:

\begin{definition}
The factor $(M_i, g_i)$ defined in Theorem~\ref{main1}, is called the {\em principal factor} of the decomposable LCP structure with respect to $g$.
\end{definition}

The above result says that the tangent bundle of the principal factor of a decomposable LCP manifold $(M,c,\nabla)$ with respect to a metric $g\in\mathcal D(M,c)$ contains the flat distribution $\mathbb{R}^q$, as well as the real line spanned by the metric dual of the Lee form of $\nabla$.
However, we will now show that it cannot be reduced to the integral manifold of the sum of these two distributions:

\begin{prop}\label{q+2}
Let $(M,c,\nabla)$ be a decomposable LCP manifold. Let $g \in \mathcal D(M,c)$ and let $(M_1,g_1)$ be the principal factor with respect to $g$. If $q$ is the dimension of the flat distribution of the LCP structure, then the dimension of $M_1$ is at least $q+2$.
\end{prop}
\begin{proof}
Throughout the proof we will use the musical isomorphisms between $1$-forms and vector fields determined either by $\tilde g$ or by $g$, depending on whether we work on $\tilde M$ or $M$.

We prove the proposition by contradiction, assuming that the dimension of $M_1$ is $q+1$. Then, the orthogonal $(T \RR^q)^\perp$ in $T M_1$ of the flat distribution is a one-dimensional distribution, and thus defines a one-dimensional distribution $D$ on the universal cover $\tilde M$ of $M$. Since $\tilde M$ is simply connected, there is a $\tilde g$-unitary vector field $\xi \in T \tilde M$ generating $D$ at each point.

On $\tilde M$, the lift $\tilde \nabla$ of $\nabla$ is the Levi-Civita connection of a metric $h = e^{2f} \tilde g$, where $\tilde \theta := df$ is the Lee form of $\tilde\nabla$ with respect to $\tilde g$. Since $g$ is adapted, one has $\tilde \theta (X) = 0$ for any $X \in T \RR^q$, giving $\tilde \theta^\sharp \in (T \RR^q)^\perp \cap T M_1 = D$, and we deduce that $\tilde\theta = d f (\xi) \xi^\flat$.

Every vector $a\in\RR^q$ determines a $\tilde\nabla$-parallel vector field $X_a \in T \RR^q$. Using the formula for conformal change of Levi-Civita connections \cite[\S 1.159]{besse}, one has
\begin{equation} \label{4.7eq1}
\nabla^{\tilde g}_{X_a} \xi =\nabla^h_{X_a} \xi - \tilde\theta ({X_a}) \xi - \tilde\theta(\xi) {X_a} + \tilde g({X_a},\xi) \tilde\theta^\sharp = \tilde\nabla_{X_a} \xi - \tilde\theta (\xi) {X_a}.
\end{equation}
Moreover, one has $\tilde g (\nabla^{\tilde g}_{X_a} \xi, \xi) = 0$, so $\nabla^{\tilde g}_{X_a} \xi \in T M_1 \cap D^\perp = T \RR^q$. Equation~\eqref{4.7eq1} thus yields $\tilde \nabla_{X_a} \xi \in T \RR^q$. On the other hand, for any vector field $Y \in T \RR^q$ one has 
$$h(\tilde \nabla_{X_a} \xi,Y)=h(\nabla^h_{X_a} \xi,Y)=-h(\xi,\nabla^h_{X_a} Y) = 0,$$
thus showing that $\tilde \nabla_{X_a} \xi = 0$. In particular,
\begin{equation}\label{bracket}
    [{X_a},\xi] = \tilde \nabla_{X_a} \xi - \tilde \nabla_\xi {X_a} = 0,
\end{equation}
and by \eqref{4.7eq1} we also have the equality
\begin{equation} \label{4.7eq2}
\nabla^{\tilde g}_{X_a} \xi = - df(\xi) {X_a}.
\end{equation}

Moreover, $\tilde g(\nabla^{\tilde g}_{\xi} \xi,\xi)=0$, and for every $a\in\RR^q$ we have by \eqref{bracket}:
$$\tilde g(\nabla^{\tilde g}_{\xi} \xi,X_a)=-\tilde g(\xi,\nabla^{\tilde g}_{\xi}X_a)=-\tilde g( \xi,\nabla^{\tilde g}_{X_a}\xi)=0,$$
thus showing that 
\begin{equation} \label{4.7eq3}
\nabla^{\tilde g}_{\xi} \xi = 0.
\end{equation}

From \eqref{4.7eq2} and \eqref{4.7eq3} we obtain that $\xi^\flat$ is a closed $1$-form. In addition, $M_1$ is simply connected, hence there exists a function $\eta \in C^\infty (M_1)$ such that $d \eta = \xi^\flat$ and $\eta (x) = 0$.

Let now $x \in \tilde M$ and let $(M_1)_x$ be the integral manifold of the distribution $T M_1$ passing through $x$, which is diffeomorphic to $M_1$. We define the map $\Psi : \RR^q \times \RR \to (M_1)_x$, $(a,t) \mapsto \psi_a^1 \circ \psi_\xi^t (x)$, where $\psi_a^1$ is the flow of $X_a$ at time $1$ and $\psi_\xi^t$ is the flow of $\xi$ at time $t$.  We also introduce $p_{\RR^q}$, the projection onto the flat part in the decomposition $\tilde M = \RR^q \times N$. We can now define the map $\Xi : (M_1)_x \to \RR^q \times \RR$, $y \mapsto (p_{\RR^q} (y), \eta(y))$. Since the two flows in the definition of $\Psi$ commute, we easily get that $\Xi$ and $\Psi$ are inverse of each other because
\begin{align*}
d\Psi \circ d\Xi = \mathrm{id}, && d\Xi \circ d\Psi = \mathrm{id}, && \Psi \circ \Xi(x) = (0,0), && \Xi \circ \Psi(0,0) = x,
\end{align*}
so $\Psi$ is a diffeomorphism and $(M_1, g_1) \simeq (\RR^q \times \RR, e^{-2f} g_{\RR^q} + dt^2)$ where $g_{\RR^q}$ is the standard metric on $\RR^q$. Under this identification, $f$ is a function of $t$, $\xi=\frac{\partial}{\partial t}$ and $\tilde\theta = f'(t) dt$.

The action of $\pi_1(M)$ can be restricted to $M_1$, since this group acts by isometries of the metric $\tilde g$. In addition, $\pi_1(M)$ also preserves the decomposition $\tilde M \simeq \RR^q \times N$ introduced in Theorem~\ref{fundLCP}, thus it preserves the two factors of the decomposition $M_1 \simeq \RR^q \times \RR$ and we denote by $H$ the restriction of $\pi_1(M)$ to the last factor $\RR$. The group $H$ acts by isometries for the metric $dt^2$, so it contains only translations or translations composed with $-\mathrm{id}$. Up to a translation, we can assume that this last map, if it lies in $H$, is $-\mathrm{id}$, and $H$ is generated by translations and $\pm \mathrm{id}$.

We know that the function $f$ is equivariant, meaning that for any $\gamma \in \pi_1(M)$ there is $c_\gamma \in \RR$ such that
\begin{align}
\gamma^* f(t) = f(t) + c_\gamma, && \forall t \in \RR.
\end{align}
Note that the constant $c_\gamma$ associated to $\pm\mathrm{id}$ is necessarily $0$ and since there exists a non-isometric homothety of $h$ in $\pi_1(M)$, $H$ must contain at least one non-trivial translation $\tau \in \RR$. Without loss of generality, we can assume that $\tau > 0$ and $\tau^* f = f + c$ with $c > 0$. The function $f$ is then entirely determined by its values on $[0, \tau]$, and if $-\mathrm{id}$ were in $H$, then $f$ would be symmetric, which is impossible because $\lim_{t \to +\infty} f(t) = +\infty$ and $\lim_{t \to - \infty} f(t) = - \infty$. Thus, $H$ contains only translations.

If $H$ is an abelian group of rank $1$, then one can assume that $\tau$ generates $H$, and $t \mapsto f(t) - \frac{c}{\tau} t$ is $H$-invariant. If $H$ is of rank at least $2$, then $H$ is a dense subgroup of $(\RR,+)$ and $f - f(0)$ is a group homomorphism between $H$ and $(\RR,+)$. The set $f(H) - f(0)$ contains $c > 0$, thus $f - f(0)$ is a non-trivial group homomorphism from $(\RR,+)$ to itself by continuity. We deduce that $f$ is an affine map. In both cases, there is $\lambda \in \RR \setminus \{0\}$ such that $f(t) - \lambda t$ is $H$-invariant and thus $\pi_1(M)$-invariant when seen as a function on $\tilde M$. Therefore, it descends to a function $\varphi$ on $M$. Notice that $\lambda \neq 0$ because $f$ is unbounded.

Using Equation~\eqref{4.7eq2} and setting a basis $(e_1,\ldots,e_q)$ of $\RR^q$, we can compute the codifferential of $\xi^\flat = dt$:
\begin{equation}
\delta^{\tilde g} \xi^\flat = - \sum_{i=1}^q e_i \lrcorner \nabla^{\tilde g}_{e_i} \xi^\flat - \xi \lrcorner \nabla^{\tilde g}_\xi \xi^\flat = q \tilde\theta(\xi).
\end{equation}
The vector field $\xi$ is preserved by the fundamental group of $M$, as we already emphasized earlier, so it descends to a vector field $\bar \xi$ on $M$, and one has:
\begin{equation}
\delta^g (e^{q\varphi} \bar \xi^\flat) = - g(\nabla^g e^{q \varphi},\bar\xi) + e^{q\varphi} \delta^g \bar\xi = (\lambda - q) e^{q\varphi} \theta(\bar\xi) + q e^{q\varphi} \theta(\bar \xi) = \lambda e^{q\varphi}.
\end{equation}
Integrating this last equality on the compact manifold $M$ and using the divergence theorem yields $\lambda = 0$, which contradicts the fact noticed above that $\lambda \neq 0$.
\end{proof}

\subsection{Lie LCP structures} We will now turn our attention to the special case of LCP manifolds whose universal covers have a Lie group structure. This special situation was recently studied in \cite{ABM} and \cite{dBM}.

\begin{definition}\label{LieLCP}
A {\em Lie LCP structure} on a compact manifold $M$ is a pair $(g, \nabla)$ such that
\begin{itemize}
\item $(M, [g], \nabla)$ is an LCP manifold;
\item if $\tilde M$ is the universal cover of $M$ and $\tilde g$ is the lift of the metric $g$, then $(\tilde M, \tilde g)$ has the structure of a Riemannian Lie group, such that the lift $\tilde \theta_g$ to $\tilde M$ of the Lee form of $\nabla$ with respect to $g$ is left-invariant, and the action of $\pi_1(M)$ on $\tilde M$ is by left-translations.
\end{itemize}
\end{definition}

\begin{definition}
A Lie LCP manifold is a triple $(M, g, \nabla)$ such that $M$ is a compact manifold and $(g, \nabla)$ is a Lie LCP structure on $M$.
\end{definition}

Note that the Lie group structure on the universal cover of a Lie LCP manifold is not necessarily unique (see below). However, for every choice of Lie group structure $G=\tilde M$, the group $G$ carries lattices, so by \cite{Milnor} it is unimodular. Its Lie algebra $\mathfrak{g}$ carries an LCP structure $(g,\mathfrak{u},\theta)$  (cf. \cite[Definition 2.1]{dBM}), where $g$ is the metric induced by the metric $g$ on $M$, $\mathfrak{u}\simeq \mathbb{R}^q$ is the flat part and $\theta$ is the linear form on $\mathfrak{g}$ induced by $\tilde \theta_g$. This structure is proper (in the sense that $\mathfrak{u}\subsetneq \mathfrak{g}$), and adapted, in the sense that $\theta|_\mathfrak{u}=0$ by \cite[Theorem 4.2]{dBM}. Conversely, \cite[Proposition 2.4]{ABM} shows that every proper LCP Lie algebra $(\mathfrak{g}, g,\mathfrak{u},\theta)$ whose associated simply connected Lie group $G$ admits a lattice $\Gamma$, determines a Lie LCP structure on $M:=\Gamma\backslash G$.

In \cite[Corollary 4.10]{dBM} it is shown that every Lie algebra $\mathfrak{g}$ carrying an adapted LCP structure $(g,\mathfrak{u},\theta)$ is a semidirect product $\mathfrak{u}\rtimes_\alpha\mathfrak{h}$, where $\mathfrak{u}$ is a flat ideal and $\mathfrak{h}$ is a non-unimodular Lie algebra acting on $\mathfrak{u}$ by a conformal representation $\alpha$. Moreover, if $\mathfrak{g}$ is unimodular, then the Lee form $\theta$ (which by assumption vanishes on $\mathfrak{u}$) satisfies 
$$\theta(x)=-\frac1{\mathrm{dim}(\mathfrak{u})}\mathrm{tr}_\mathfrak{h}\mathrm{ad}_x,\qquad\forall x\in\mathfrak{h}.$$

\begin{example} \label{FundamentalEx} Let $(\mathfrak{h},g_h)$ be the 2-dimensional Lie algebra determined by an orthonormal basis $e_0,e_1$ satisfying $[e_0,e_1]=e_1$, and let $\alpha$ be the representation of $\mathfrak{h}$ on $\mathbb{R}$ given by $\alpha(e_0)=0$ and $\alpha(e_1)=-\mathrm{Id}_\RR$. Then $\mathfrak{g}:=\mathbb{R}\rtimes_\alpha\mathfrak{h}$ can also be written as $\RR^2 \rtimes \RR$, where the second factor acts on $\RR^2$ via the representation
\[
t \mapsto t \left( \begin{matrix} 1 & 0 \\ 0 & -1 \end{matrix} \right) =: t A_0.
\]

The simply connected Lie group $G$ with Lie algebra $\mathfrak{g}$
can be identified to the semi-direct product $\RR^2 \rtimes \RR$ endowed with the product
\begin{align}\label{gr}
(x,t) \cdot (x', t') = (x + e^{tA_0} x', t + t'), && \forall (x,t),(x',t') \in G.
\end{align}
Let $\lambda > 0$ be a root of the polynomial $P(X) := X^2 - 3 X + 1$. Setting $t_0 = \ln(\lambda)$, the matrix  $e^{t_0 A_0}$ has eigenvalues $\lambda$ and $\lambda^{-1}$ and its characteristic polynomial is $P$. Consequently, $e^{t_0 A_0}$ is conjugate to the matrix
\[
A = \left( \begin{matrix} 1 & 1 \\ 1 & 2 \end{matrix} \right),
\]
i.e. there exists $Q \in \mathrm{GL}_2 (\RR)$ such that $Q^{-1} A Q = e^{t_0 A_0}$. Let $(v_1, v_2) \in (\RR^2)^2$ be the column vectors of $Q^{-1}$ and let $\langle v_1, v_2 \rangle\simeq \mathbb{Z}^2$ and $\langle t_0 \rangle\simeq \mathbb{Z}$ denote the corresponding translation groups of $\mathbb{R}^2$ and $\mathbb{R}$ respectively. Then, the subgroup $\Gamma:= \langle v_1, v_2 \rangle \rtimes \langle t_0 \rangle$ of $G$ is a lattice in $G$, thus $M :=  \Gamma \backslash G$ is compact. We now denote by $\tilde g$ the left-invariant metric on $G$ determined by the standard metric on the Lie algebra $\mathfrak{g}\simeq\RR^3$. A straightforward computation shows that in the coordinates $((x,y),t) \in \RR^2 \times \RR$, this metric is given by
\begin{equation}
\tilde g = e^{-2 t} dx^2 + e^{2 t} dy^2 + dt^2,
\end{equation}
and it descends to a metric $g$ on $M$. The first two generators $v_1,v_2$ of the group $H$ act isometrically by left-translation on $(G, h := e^{2 t} \tilde g)$, whereas the third generator $t_0$ acts on $(G, h)$ as a strict homothety of ratio $\lambda^2$. The Levi-Civita connection of $h$ thus descends to a closed, non-exact Weyl connection $\nabla$ on $M$, which is reducible and non-flat, defining an LCP structure on $M$ (see \cite{FlaLCP} for similar but more general constructions of LCP manifolds). The LCP manifold $(M, [g], \nabla)$ is precisely the example introduced in \cite{MN15}. The lift of the Lee form of $\nabla$ with respect to $g$ is $dt$, which is left-invariant, and $\pi_1(M)=\Gamma$ is a subgroup of $G$, showing that $(M, g, \nabla)$ is a Lie LCP manifold.
\end{example}

\subsection{Decomposable Lie LCP manifolds}

In this section, we derive some direct corollaries about Lie LCP structures arising from the previous considerations. We then prove that the example of non-solvable Lie LCP manifold constructed in \cite[\S5.2]{dBM} is indecomposable. This later fact reinforces the interest of this example. 

Indeed, constructing decomposable Lie LCP manifolds which are not solvmanifolds is trivial, simply by taking an LCP solvmanifold (whose invariant metric is automatically adapted) and making the Riemannian product with the quotient of a semi-simple Riemannian Lie group by a co-compact lattice (which always exists). However, it is much more difficult to construct lattices in indecomposable LCP Lie algebras whose semi-simple part is of non-compact type, and this is the reason why it is important to check that the example in \cite[\S5.2]{dBM} is indecomposable as LCP manifold.

Our first remark is that, by the above results, in the Lie LCP setting we can remove all ambiguities in the definition of decomposable LCP manifold, since there is at most one reducible metric in the conformal class:

\begin{cor} \label{decompLieLCP}
Let $(M,g,\nabla)$ be a decomposable Lie LCP manifold. Then, $g$ is the only reducible metric in $[g]$ up to a multiplicative constant.
\end{cor}
\begin{proof}
This follows directly from Theorem~\ref{simplyconnectedHspace} applied to the universal cover $G$ of $M$ endowed with its Lie group structure and the lift $\tilde g$ of $g$.
\end{proof}

This corollary is remarkable, because it does not hold for general LCP manifolds. We give here a counter-example:

\begin{example}
Consider a reducible LCP manifold $(M_1 \times M_2, [g_1 + g_2], \nabla^g + \overline{\pi_1^* \theta})$ as introduced in Example~\ref{reducibleLCP}, with $M_2 := S^1 \times S^1$. Let $(X_1,X_2)$ be the canonical left-invariant basis of $T(S^1 \times S^1)$ and let $(\eta_1, \eta_2)$ be its dual frame. We take any non-constant function $\bar f : S^1 \to \RR$ and we define $f := p_2^*(\bar f)$ where $p_2$ is the projection on the second factor of $S^1 \times S^1$. With these notations, we take the metric $g_2$ to be
\begin{equation}
g_2 := e^{2f} \eta_1^2 + \eta_2^2.
\end{equation}
Then, the two metrics
\begin{align}
g_1 + g_2 = g_1 + e^{2f} \eta_1^2 + \eta_2^2 && \textrm{and} && e^{-2f} (g_1 + g_2) = e^{-2f} (g_1 + \eta_2^2) + \eta_1^2
\end{align}
belong to the conformal class $[g_1 + g_2]$, are non-homothetic and are both reducible.
\end{example}

As a corollary of the general result for Lie groups, we give a necessary and sufficient condition of decomposability for Lie LCP manifolds in terms of the metric on the Lie algebra.

\begin{cor}\label{declcp}
Let $(M, g, \nabla)$ be a Lie LCP manifold. Denote by $(\mathfrak g, \langle \cdot, \cdot\rangle,\theta,\mathfrak{u})$ the corresponding LCP Lie algebra. Then, $(M,g,\nabla)$ is decomposable if and only if there exists a non-trivial orthogonal decomposition $\mathfrak g = \mathfrak g_1 \overset{\perp}{\oplus} \mathfrak g_2$ such that $\mathfrak g_1$ and $\mathfrak g_2$ are Lie subalgebras of $\mathfrak g$ satisfying \eqref{cond} and $\mathfrak{u}+\theta^\sharp\subset\mathfrak{g}_1$.
\end{cor}
\begin{proof}
Direct consequence of Proposition \ref{redg} and \cite[Theorem 4.7]{BFM}.
\end{proof}

As an application, we will show the indecomposability of the Lie LCP structures determined by the non-solvable LCP Lie algebra constructed in \cite[\S5.2]{dBM}. We briefly recall this construction.

For $d\geq 2$ we set $n:=d^2$ and identify $\mathbb{R}^n$ with the set of $d\times d$ matrices. Right multiplication defines 
a Lie algebra representation $\rho:\mathfrak{sl}(d,\mathbb{R})\to  \mathfrak{gl}(n,\mathbb{R})$. 

Let $A$ be the $2\times 2$ diagonal matrix $A:=\mathrm{diag}(1,-1)$. We define two Lie algebra representation $\tau_1:\mathfrak{sl}(d,\mathbb{R})\times \mathbb{R}\to \mathfrak{gl}(n+1,\mathbb{R})$ and $\tau_2:\mathfrak{sl}(d,\mathbb{R})\times \mathbb{R} \to \mathfrak{gl}(2,\mathbb{R})$  by 
\[
 \tau_1(M,t)={\rm diag}(\rho(M),0),\qquad \tau_2(M,t)=tA, \qquad \forall M\in \mathfrak{sl}(d,\mathbb{R}),\ \forall t\in\mathbb{R},
\]
where the first map is block diagonal using the inclusion $\mathfrak{gl}(n,\mathbb{R})\subset \mathfrak{gl}(n+1,\mathbb{R})$. Their tensor product gives rise to the representation 
\[
\tau:=\tau_1\otimes \tau_2:\mathfrak{sl}(d,\mathbb{R})\times \mathbb{R} \to \mathfrak{gl}(\mathbb{R}^{n+1}\otimes\mathbb{R}^2)=\mathfrak{gl}(2n+2,\mathbb{R})
\]
which is explicitly defined by
\begin{equation}
\label{tau}
\tau(M,t)=\mathrm{diag}(\rho(M),0)\otimes \mathrm{Id}_2+t \,\mathrm{Id}_{n+1}\otimes A, \qquad \forall M\in\mathfrak{sl}(d,\mathbb{R}), \  \forall t\in\mathbb{R}.
\end{equation}

In \cite[Proposition 5.2]{dBM} it is shown that the Lie algebra obtained as a semidirect product 
\[\mathfrak{g}=(\mathbb{R}^{n+1}\otimes \mathbb{R}^2)\rtimes_{\tau} (\mathfrak{sl}(d,\mathbb{R})\oplus \mathbb{R} b)
\]
carries an LCP structure and its corresponding simply connected Lie group $G$ carries a lattice $\Gamma$, thus defining a Lie LCP structure on the manifold $M:=\Gamma\backslash G$. More precisely, one considers the inner product $g$ on $\mathfrak{g}$ making the factors $\mathfrak{sl}(d,\mathbb{R})$, $\mathbb{R} b$, $\mathbb{R}^{n+1}\otimes \mathbb{R}^2$ orthogonal, $b$ of norm 1 and such that, when restricted to $\mathbb{R}^{n+1}\otimes \mathbb{R}^2$, it is the tensor product of the canonical inner products on each factor. If $\{e_i\}_{i=1}^{n+1}$ and $\{v_1,v_2\}$ denote the canonical bases of $\mathbb{R}^{n+1}$ and $\mathbb{R}^2$, respectively, the flat space is $\mathfrak{u}:=\mathbb{R} e_{n+1}\otimes v_1$ and the Lee form is the metric dual of $b$. 
 
According to \cite{dBM}, the Lie algebra $\mathfrak{g}$ is indecomposable (i.e. it is not a direct sum of proper ideals). However, the corresponding Lie LCP structure could still be decomposable. The next result shows that this is not the case.

\begin{prop} \label{notdecomposable}
    The Lie LCP manifold $M$ constructed above is non-decomposable.
\end{prop}
\begin{proof}
    In view of Corollary \ref{declcp}, we need to show that if $\mathfrak g = \mathfrak g_1 {\oplus} \mathfrak g_2$ is an orthogonal decomposition such that $\mathfrak g_1$ and $\mathfrak g_2$ are Lie subalgebras of $\mathfrak g$ satisfying \eqref{cond} and such that $\mathfrak{u}+\theta^\sharp\subset\mathfrak{g}_1$, then $\mathfrak{g}_2=0$. 
    
    Let $\mathfrak g = \mathfrak g_1 {\oplus} \mathfrak g_2$ be such an orthogonal decomposition. Then the vectors  $b$ and $e_{n+1}\otimes v_1$ belong to $\mathfrak{g}_1$. Applying \eqref{cond2} to $y_1:=b$, and using the fact that $\theta$ is closed, so $b$ is orthogonal to $\mathfrak{g}'$, we obtain $\mathrm{ad}_b(\mathfrak{g}_2)\subset \mathfrak{g}_2$.  

    Using \eqref{tau} for $M=0$ and $t=1$ we see that $\mathrm{ad}_b$ has three eigenspaces: $\mathfrak{sl}(d,\mathbb{R})\oplus \mathbb{R} b$ for the eigenvalue 0, $\mathbb{R}^{n+1}\otimes v_1$ for the eigenvalue 1 and $\mathbb{R}^{n+1}\otimes v_2$ for the eigenvalue $-1$. Consequently, $\mathfrak{g}_2$ is a direct sum of three vector subspaces $\mathfrak{g}_2=E_0\oplus E_1\oplus E_2$, with $E_0\subset\mathfrak{sl}(d,\mathbb{R})\oplus \mathbb{R} b$,  $E_1\in \mathbb{R}^{n+1}\otimes v_1$, and $E_2\subset \mathbb{R}^{n+1}\otimes v_2$.

    Applying \eqref{cond} to $x_1=b$ and $x_2\in E_1$ we obtain 
    \[0=\langle [b,x_2], x_2 \rangle=\langle x_2, x_2 \rangle,\]
    thus showing that $E_1=\{0\}$. Similarly, taking $x_2\in E_2$ yields 
    \[0=\langle [b,x_2], x_2 \rangle=-\langle x_2, x_2 \rangle,\]
    so $E_2=\{0\}$ as well. Moreover $\mathfrak{g}_2$ is orthogonal to $b$ so it is contained in $\mathfrak{sl}(d,\mathbb{R})$. 

    We apply \eqref{cond} again, this time to an arbitrary element $x_2=M\in \mathfrak g_2 \subset \mathfrak{sl}(d,\mathbb{R})$ and $x_1=N\otimes v_1\in \mathbb{R}^n\otimes v_1\subset \mathfrak{g}_1$, where $N\in \mathbb{R}^n$ is identified to a $d\times d$ matrix. We get by \eqref{tau} for $t=0$:
    \[0=\langle [M,x_1], x_1 \rangle=\langle \rho(M)N\otimes v_1, N\otimes v_1 \rangle=\langle NM, N\rangle=\langle M, N^*N\rangle,
    \]
    showing that $M$ is orthogonal to $\mathrm{Sym}^2(\mathbb{R}^d)$ (because the set of matrices of the form $N^*N$ generates the vector space $\mathrm{Sym}^2(\mathbb{R}^d)$ of symmetric matrices). Since $M$ was an arbitrary element of $\mathfrak{g}_2$, this implies that $\mathfrak{g}_2$ is orthogonal to $\mathrm{Sym}^2(\mathbb{R}^d)$, thus
    $\mathfrak{g}_2\subset \mathfrak{so}(d,\mathbb{R})$. By orthogonality we then have $\mathrm{Sym}^2(\mathbb{R}^d)\subset \mathfrak{g}_1$, and since $[\mathrm{Sym}^2(\mathbb{R}^d),\mathrm{Sym}^2(\mathbb{R}^d)]=\mathfrak{so}(d,\mathbb{R})$, and $\mathfrak{g}_1$ is a subalgebra, we conclude that $\mathfrak{g}_1\supset\mathfrak{so}(d,\mathbb{R})$. On the other hand we have seen that $\mathfrak g_2$ is contained in $\mathfrak{so}(d,\mathbb{R})$ and is orthogonal to $\mathfrak{g}_1$. This shows that $\mathfrak g_2=0$, thus finishing the proof.
\end{proof}

\section{Weak reducibility and strong irreducibility}

As we emphasized before, the main issue with the notion of reducibility as introduced in Example~\ref{reducibleLCP} is that it does not behave well with respect to the action of the fundamental group on the universal cover. Nevertheless, in our seek for simplest LCP manifolds, we consider a weaker version of reducibility allowing one to partially overcome this obstacle.

\begin{definition} \label{weakreducible}
    Let $(M,c,\nabla)$ be a decomposable LCP manifold. Then, we say that $(M,c,\nabla)$ is {\em weakly reducible} if there is a metric $g \in \mathcal D(M,c)$ with principal factor $(M_1,g_1)$ such that the group $\Gamma_1:=\pi_1(M) \vert_{M_1}$ acts freely and properly discontinuously on $M_1$.
\end{definition}

The motivation of the previous definition is that, as we shall see in Proposition \ref{weakreducibleprop} below, $\Gamma_1\backslash M_1$ is a compact manifold which inherits an LCP structure. Thus, if $(M,c,\nabla)$ is weakly reducible, then even though the underlying Riemannian manifold $(M,g)$ is not necessarily a Riemannian product with a compact LCP factor, like in the reducible case, it is still obtained from a smaller compact LCP manifold by a mapping torus-like construction.

\begin{prop} \label{weakreducibleprop}
    Let $(M,c,\nabla)$ be a weakly reducible LCP manifold. Then, in the notation of Definition~\ref{weakreducible}, $M' := \Gamma_1\backslash M_1$ is a compact manifold, $g_1:=g|_{M_1}$ descends to a metric $g'$ on $\bar M$, and the triple $(M', [g'],
    \nabla \vert_{M_1})$ is an LCP manifold.
\end{prop}
\begin{proof}
We start by proving that $M'$ is a compact manifold. The group $\Gamma_1$ acts freely and properly discontinuously on $M_1$. Thus, $\Gamma_1 \backslash M_1$ is a manifold. Since $M$ is compact, there exists a compact set $K \subset \tilde M$ such that $\pi_1(M) \cdot K = \tilde M$. The projections $K_1$ and $K_2$ of $K$ on $M_1$ and $M_2$ respectively are compact. One has $K \subset K_1 \times K_2$ and $\pi_1(M) \cdot (K_1 \times K_2) = \tilde M$. This implies that $\pi_1 (M) \vert_{M_1} \cdot K_1 = M_1$, yielding $\Gamma_1 \cdot K_1 = M_1$, so the action of $\Gamma_1$ on $M_1$ is co-compact.

The pull-back $\tilde \theta$ of $\theta$ to $\tilde M$ and the symmetric tensor $g_1$ are both invariant under the action of $\pi_1(M)$ and thus invariant under the action of $\Gamma_1$ when seen as objects on $M_1$. Moreover, $g_1:=g|_{M_1}$ is clearly the pull-back of a metric on $M_1$ by the projection $p_1:\tilde M\to M_1$, and by \cite[Theorem 4.7]{BFM}, $\tilde \theta$ is also the pull-back of a $1$-form on $M_1$ by $p_1$.
Consequently, they descend respectively to a metric $g'$ and a $1$-form $\theta'$ on $\Gamma_1\backslash M_1$.

Since  $\tilde \theta$ is exact on $\tilde M$, it has to be the pull-back of an exact $1$-form on $M_1$, thus there exists $f \in C^\infty(M_1)$ such that $\tilde \theta = d f$. Moreover, the lift $\tilde \nabla$ of $\nabla$ to $\tilde M$ is the Levi-Civita connection of $e^{2f} \tilde g$. Consequently, $(M_1, e^{2f}g_1)$ is a totally geodesic submanifold of the warped product $(M_1 \times M_2, e^{2f} g_1 + e^{2f} g_2)$, thus $\nabla^{e^{2f} g_1} = \nabla\vert_{M_1}$. In particular, $(M_1, e^{2f} g_1)$ has reducible holonomy, because the flat part of the LCP manifold $(M,c,\nabla)$ is still parallel with respect to the connection $\nabla \vert_{M_1}$. The fundamental group $\Gamma_1$ of $M'$ acts by homotheties on $(M_1, e^{2f} g_1)$, thus the connection $\nabla \vert_{M_1}$ descends to a connection on $M'$, which coincides with $\nabla^{g'} + \bar{\theta'}$. We conclude that $(M', [g'], \nabla^{g'} + \bar{\theta'} = \nabla \vert_{M_1})$ is a compact conformal manifold endowed with a connection with reducible holonomy. Moreover, the universal cover $(M_1, e^{2f} g)$ has a flat complete Riemannian factor, thus it is an LCP manifold (see \cite[Remark 2.6]{FlaLCP}).
\end{proof}

An example of a weakly irreducible LCP manifold is given in \cite[Example 4.11]{BFM}. It can be written under the form of a Lie LCP manifold.

\begin{example}
We consider the Lie algebra $\mathfrak g' := (\RR^2 \rtimes_{A_0} \RR) \times \RR$, where $A_0 = \mathrm{Diag}(1,-1)$. The simply connected group with Lie algebra $\mathfrak g'$ is $G' = G \times \RR$, where $G = \RR^{2} \rtimes \RR$ is the group constructed in Example \ref{FundamentalEx} with group law given by \eqref{gr}. We will use the coordinates $(x,y,t,s)$ on $G'$.

We consider the standard scalar product on $R^4 \simeq \mathfrak g'$, which induces a left-invariant metric $\tilde{g'}$ on $G'$. Using the notations of Example \ref{FundamentalEx}, we consider the subgroup
\[
\Gamma' := \langle v_1, v_2, (0,0,t_0,1), (0,0,0,\sqrt{2})\rangle,
\]
which is a lattice in $G'$. The metric
\begin{equation}
\tilde g' = e^{-2t} dx^2 + e^{2t} dy^2 + dt^2 + ds^2
\end{equation}
descends to a metric $g'$ on $\Gamma'\backslash G'$ and the Levi-Civita connection of $h' := e^{2 t} \tilde g$ descends to a connection $\nabla'$. It is straightforward to check that $(\Gamma' \backslash G', g', \nabla)$ is a Lie LCP manifold, and it is decomposable because $(G', \tilde{g'}) = (G,\tilde g) \times (\RR, ds^2)$. 

The lattice $\Gamma'$ is not a product of two subgroups acting separately on each factor of this product. However, the restriction of $\pi_1(M)$ to $G$ is exactly the group $\Gamma$ defined in Example~\ref{FundamentalEx}, so it acts freely and properly discontinuously, thus the LCP manifold is weakly irreducible.
\end{example}

The natural question coming with this new notion is: can we construct decomposable LCP manifolds which are not weakly reducible? We first give a name to such a manifold.

\begin{definition}\label{strongirr}
A decomposable LCP manifold which is not weakly reducible is called {\em strongly irreducible}. 
\end{definition}

The existence of strongly irreducible LCP manifolds is related to the existence of algebraic units of modulus $1$ which are not roots of unity. We give here an example:

\begin{example}
We consider the polynomial
\begin{equation}
P(X) = (X^2 - \frac{3 + \sqrt{5}}{2} X + 1) (X^2 - \frac{3 - \sqrt{5}}{2} X + 1) = X^4 - 3 X^3 + X^2 - 3 X + 1 \in \ZZ[X].
\end{equation}
The polynomial $(X^2 - \frac{3 - \sqrt{5}}{2} X + 1)$ has negative discriminant, thus its roots are two conjugated complex numbers of modulus $1$, that we denote by $e^{i \mu}$ and $e^{-i \mu}$, with $\mu \in \RR$. The polynomial $(X^2 - \frac{3 + \sqrt{5}}{2} X + 1)$ has a positive discriminant, thus it has two positive real roots that we denote by $\lambda$ and $\lambda^{-1}$, different from $\pm 1$. The matrix
\begin{equation}
    A = \left( \begin{matrix}
        \lambda & & & \\ & \lambda^{-1} & & \\ & & \cos(\mu) & - \sin(\mu) \\ & & \sin(\mu) & \phantom{-}\cos(\mu)
    \end{matrix} \right)
\end{equation}
is therefore similar to the companion matrix of $P$, so there exists $C \in \mathrm{GL}(4,\RR)$ such that $C A C^{-1} \in \mathrm{GL}(4,\ZZ)$. A logarithm of $A$ (i.e. a matrix $A_0$ such that $\exp(A_0)=A$) is given by
\begin{equation}
    A_0 = \left( \begin{matrix}
        \ln (\lambda) & & & \\ & - \ln(\lambda) & & \\ & & 0 & - \mu \\ & & \mu & \phantom{-}0
    \end{matrix} \right).
\end{equation}

We define the almost Abelian Lie algebra $\mathfrak g := \RR^4 \rtimes_{A_0} \RR$ and we endow it with the canonical scalar product of $\RR^5$. The simply connected Lie group with Lie algebra $\mathfrak g$ is $G := \RR^4 \rtimes \RR$ with the product
\begin{align}
    (x,t) \cdot (x',t') = (x + e^{t A_0} x', t + t'), && \forall (x,t),(x',t') \in G \times G.
\end{align}
It is easy to show that the left-invariant metric induced on $G$ is
\begin{equation}
    \tilde g = e^{-2 \ln(\lambda)t} dx_1^2 + e^{2 \ln(\lambda)t} dx_2^2 + dx_3^2 + dx_4^2 + dt^2.
\end{equation}
Let $v_1, v_2, v_3, v_4$ be the column vectors of $C^{-1}$. The subgroup $\Gamma := \langle v_1, v_2, v_3, v_4 \rangle \rtimes \langle 1 \rangle$ is a lattice of $G$ and the metric $\tilde g$ descends to $M := \Gamma \backslash G$. The metric $(G, h := e^{2 \ln(\lambda) t} \tilde g)$ has a complete flat Riemannian factor $\RR$ and its Levi-Civita connection descends to a closed non-exact Weyl connection on $(M, [g])$, thus $(M, g, \nabla)$ is a Lie LCP manifold.

The LCP manifold defined this way is decomposable because its universal cover admits the decomposition as Riemannian product
\[
(G,\tilde g) = (\RR^3, e^{-2 \ln(\lambda)t} dx_1^2 + e^{2 \ln(\lambda)t} dx_2^2 + dt^2) \times (\RR^2, dx_3^2 + dx_4^2).
\]
If this manifold were weakly reducible according to this decomposition, then $(\RR^3, dx_1^2 + e^{4 \ln(\lambda)t} dx_2^2 + e^{2 \ln(\lambda)t} dt^2)$ would be the universal cover of an LCP manifold endowed with the unique metric (up to a multiplicative constant) whose Levi-Civita connection is the lift of the Weyl connection by Proposition~\ref{weakreducibleprop}. The fundamental group of this LCP manifold would then act by homotheties, and the group of the similarity ratios would be generated by $\lambda$, which is an algebraic unit of degree $4$. However, this LCP manifold would be of dimension $3$, but all such LCP manifolds are classified (see \cite[Theorem 1.8]{Kou}) and have a group of homothety ratios containing only algebraic units of degree at most $2$, which is a contradiction. This example is thus strongly irreducible.
\end{example}

We can summarize the above considerations as follows. We have defined four classes of LCP structures:
\begin{enumerate}
    \item[$\mathcal{R}$]: Reducible LCP manifolds (Example \ref{reducibleLCP});
  \item[$\mathcal{WR}$]: weakly reducible LCP manifolds (Definition \ref{weakreducible});
   \item[$\mathcal{SI}$]: strongly irreducible LCP manifolds (Definition \ref{strongirr});
 \item[$\mathcal{D}$]: decomposable LCP manifolds (Definition \ref{defdec}).
\end{enumerate}

The examples above show that all these classes are non-empty and the following inclusion relations hold:
\[\mathcal{R}\subsetneq\mathcal{WR}\subsetneq\mathcal{D}=\mathcal{WR}\sqcup\mathcal{SI}.\]
	
\renewcommand{\refname}

{\bf References}
\begin{thebibliography}{999}

\bibitem{AK78}
D. V. Alekseevskii, B. N. Kimel'fel'd, {\sl Classification of homogeneous conformally flat Riemannian manifolds}. Mathematical Notes, {\bf 24} (1), 559--562 (1978).

\bibitem{ABM} A. Andrada, V. del Barco, A. Moroianu, {\sl Locally conformally product structures on solvmanifolds}. To be published in Ann. Mat. Pura Appl. doi: 10.1007/s10231-024-01449-9.

\bibitem{BFM} F. Belgun, B. Flamencourt, A. Moroianu, {\sl Weyl structures with special holonomy on compact conformal manifolds}. arXiv:2305.06637 (2023).

\bibitem{BM} F. Belgun, A. Moroianu, {\sl On the irreducibility of locally metric connections.} J. reine angew. Math. {\bf 714}, 123--150 (2016).

\bibitem{besse}  A.\ Besse, {\it Einstein manifolds}. Ergebnisse der Mathematik und ihrer Grenzgebiete (3)  10.\ Springer-Verlag, Berlin, 1987.

\bibitem{dBM} V. del Barco, A. Moroianu, {\sl The structure of locally conformally product Lie algebras}. arxiv:2404.17956 (2024).

\bibitem{FlaLCP} B. Flamencourt, {\sl Locally conformally product structures}. Internat. J. Math. {\bf 35} (5), 2450013 (2024).

\bibitem{F24}  B. Flamencourt, {\sl The characteristic group of conformally product structures.} arXiv:2401.08372 (2024).

\bibitem{Kou} M. Kourganoff, {\sl Similarity structures and de Rham decomposition}. Math. Ann. {\bf 373}, 1075--1101 (2019).

\bibitem{MN15} V.S. Matveev, Y. Nikolayevsky, {\sl A counterexample to Belgun–Moroianu conjecture}. C. R. Math. Acad. Sci. Paris {\bf 353} , 455--457 (2015).

\bibitem{MN17} V. Matveev, Y. Nikolayevsky, {\sl Locally conformally Berwald manifolds and compact quotients of
reducible manifolds by homotheties.} Ann. Inst. Fourier (Grenoble) {\bf 67} (2), 843--862 (2017).

\bibitem{Milnor} J. Milnor, {\sl Curvatures of left-invariant metrics on Lie groups.} Adv. Math. {\bf 21}, 293--329 (1976).

\bibitem{MP} A. Moroianu, M. Pilca, {\sl Adapted metrics on locally conformally product manifolds.} Proc. Amer. Math. Soc. {\bf 152}, 2221--2228 (2024).

\bibitem{MS2008} A. Moroianu, U. Semmelmann, {\sl Twistor forms on Riemannian products}. J. Geom. Phys. {\bf 58}, 1343--1345 (2008). 

\bibitem{TM67} Y.\ Tashiro, K.\ Miyashita, {\sl Conformal transformations in complete product Riemannian manifolds}. J.\ Math.\ Soc.\ Japan {\bf 19}, 328--346 (1967).

\end{thebibliography}
\end{document}